\documentclass[12pt]{article}
\usepackage{amsmath,amssymb,amsthm}

\newtheorem{theorem}{Theorem}
\newtheorem{proposition}[theorem]{Proposition}

\newtheorem{corollary}[theorem]{Corollary}

\def\barr{\begin{array}}
\def\earr{\end{array}}

\title{Breaking points in the poset of conjugacy classes of subgroups of a finite group}
\author{Marius T\u arn\u auceanu}
\date{February 2, 2018}

\begin{document}

\maketitle

\begin{abstract}
    In this note, we determine the finite groups whose poset of conjugacy classes of
    subgroups has breaking points. This leads to a new characterization of the
    generalized quaternion $2$-groups. A generalization of this property is also studied.
\end{abstract}

{\small
\noindent
{\bf MSC 2010\,:} Primary 20D30; Secondary 20D15, 20E15.

\noindent
{\bf Key words\,:} breaking point, poset of conjugacy classes of subgroups, interval, generalized quaternion $2$-group.}

\section{Introduction}

Let $G$ be a finite group and $L(G)$ be the subgroup lattice of $G$. The starting point for our discussion is given by \cite{2},
where the proper nontrivial subgroups $H$ of $G$ satisfying the condition
$$\text{for every}\hspace{1mm} X\in L(G)\hspace{1mm} \text{we have either}\hspace{1mm} X\leq H\hspace{1mm} \text{or}\hspace{1mm} H\leq X \leqno(1)$$have been studied. Such a subgroup is called a \textit{breaking point} for the lattice $L(G)$, and a group $G$ whose subgroup lattice possesses breaking points is called a \textit{BP-group}. Clearly, all cyclic $p$-groups of order at least $p^2$ and all generalized quaternion $2$-groups $Q_{2^n}=\langle a,b \,|\, a^{2^{n-2}}= b^2, a^{2^{n-1}}=1, b^{-1}ab=a^{-1}\rangle$, $n\geq 3$, are BP-groups. Note that a complete classification of BP-groups can be found in \cite{2}. Also, we observe that the condition (1) is equivalent to $$L(G)=[1,H]\cup[H,G],\leqno(2)$$where for $X,Y\in L(G)$ with $X\subseteq Y$ we denote by $[X,Y]$ the interval in $L(G)$ between $X$ and $Y$. A natural generalization of (2) has been suggested by Roland Schmidt, namely
$$L(G)=[1,M]\cup[N,G] \mbox{ with } 1<M,N<G,\leqno(3)$$and the abelian groups $G$ satisfying (3) have been determined in \cite{1}.

The above concepts can be naturally extended to other remarkable posets of subgroups of $G$, such us the posets $C(G)$ and $\overline{C}(G)$ of cyclic subgroups and of conjugacy classes of cyclic subgroups of $G$, respectively. We recall here that the generalized quaternion $2$-groups
$Q_{2^n}$, $n\geq 3$, can be characterized as being the unique finite non-cyclic groups $G$ for which $C(G)$ and $\overline{C}(G)$ have breaking points (see \cite{7} and \cite{3}, respectively).

In the current note, we will focus on the set of conjugacy classes of subgroups of $G$. It is defined by $$\overline{L}(G)=\{[H] \,|\, H\in L(G)\}, \mbox{ where } [H]=\{H^x\,|\, x\in G\},\, \forall\, H\in L(G).$$Under the ordering relation
$$[H_1]\leq [H_2]\mbox{ if and only if } H_1\subseteq H_2^x \mbox{ for some } x\in G$$this is a poset with the least element $[1]=\{1\}$ and the greatest element $[G]=\{G\}$. We will prove that the cyclic $p$-groups of order at least $p^2$ and the generalized quaternion $2$-groups are the unique finite groups $G$ for which $\overline{L}(G)$ has breaking points. The more general problem of finding the finite groups $G$ such that $\overline{L}(G)$ is a union of two proper intervals will be also addressed.

Most of our notation is standard and will usually not be repeated here. Elementary concepts and results on group theory
can be found in \cite{4,6}. For subgroup lattice notions we refer the reader to \cite{5,8} .

\section{Main results}

Our main theorem is the following.

\begin{theorem}
    Let $G$ be a finite group and $\overline{L}(G)$ be the poset of conjugacy classes of subgroups of $G$. Then
    $\overline{L}(G)$ possesses breaking points if and only if $G$ is either a cyclic $p$-group of order at least
    $p^2$ or a generalized quaternion $2$-group.
\end{theorem}

\begin{proof}
    Clearly, if $G$ is a cyclic $p$-group of order at least $p^2$, then $\overline{L}(G)$ possesses breaking points. Also,
    if $G$ is a generalized quaternion $2$-group, then $G$ has exactly one subgroup of order $p$ (see e.g. (4.4) of \cite{6}, II), and the conjugacy class of this subgroup is a breaking point for $\overline{L}(G)$.

    Conversely, assume that $\overline{L}(G)$ possesses a breaking point, say $[H]$. Then $G$ is a $p$-group. Indeed, if the order of $G$ has at least two distinct prime divisors, then $H$ cannot be a $p$-subgroup. It follows that $[S]\leq [H]$ for any Sylow subgroup $S$ of $G$. This leads to $H=G$, a contradiction.

    Let $|G|=p^n$ and $|H|=p^m$, where $p$ is a prime and $1\leq m\leq n-1$. Then for every subgroup $K\leq G$ of order $p^m$ we have either $[K]\leq [H]$ or $[H]\leq [K]$, implying that $[K]=[H]$. This shows that all subgroups of order $p^m$ of $G$ are conjugate to $H$. On the other hand, since $G$ has a normal subgroup of order $p^m$, we infer that $H$ is in fact the unique subgroup of order $p^m$ of $G$, i.e. a breaking point for $L(G)$. Thus $G$ is a BP-group, and by Theorem 1.1 of \cite{2} it follows that $G$ is either a cyclic $p$-group of order at least $p^2$ or a generalized quaternion $2$-group. This completes the proof.
\end{proof}

The following corollary is obtained directly from Theorem 1.

\begin{corollary}
    The generalized quaternion $2$-groups are the unique finite noncyclic groups whose
    poset of conjugacy classes of subgroups has breaking points.
\end{corollary}

By bringing together Theorem 1 and the main results of \cite{2,3,7}, one obtains:

\begin{corollary}
    Let $G$ be a finite group and $L(G)$, $\overline{L}(G)$, $C(G)$, $\overline{C}(G)$ be the posets of subgroups, of conjugacy classes of subgroups, of cyclic subgroups, and of conjugacy classes of cyclic subgroups of $G$, respectively. Then the following conditions are equivalent:
    \begin{itemize}
    \item[{\rm a)}] $L(G)$ has breaking points.
    \item[{\rm b)}] $\overline{L}(G)$ has breaking points.
    \item[{\rm c)}] $C(G)$ has breaking points.
    \item[{\rm d)}] $\overline{C}(G)$ has breaking points.
    \end{itemize}
\end{corollary}
\medskip

In what follows we will denote by $\cal C$ the class of finite groups $G$ satisfying $$\overline{L}(G)=[[1],[M]]\cup[[N],[G]] \mbox{ with } 1<M,N<G.\leqno(4)$$\newpage

\noindent Clearly, (4) is equivalent with (3) for abelian groups $G$ because in this case we have $\overline{L}(G)=L(G)$. Then all finite abelian groups satisfying (3) belong to $\cal C$. So, we may restrict our attention only to finite non-abelian groups contained in $\cal C$. By Theorem 1, we know that the generalized quaternion $2$-groups $Q_{2^n}$, $n\geq 3$, have this property. We also observe that there are many examples of groups of small order contained in $\cal C$, such us the symmetric group $S_3$ (for which we can choose $M=\langle(1\,2)\rangle$ and $N=\langle(1\,2\,3)\rangle$), the dihedral group $D_{10}=\langle x,y\,|\,x^5=y^2=1, x^y=x^{-1}\rangle$ (for which we can choose $M=\langle x\rangle$ and $N=\langle y\rangle$), and the alternating group $A_4$ (for which we can choose $M=\{(1),(1\,2)(3\,4),(1\,3)(2\,4),(1\,4)(2\,3)\}$ and $N=\langle(1\,2\,3)\rangle$).

First of all, we will focus on the finite $p$-groups $G$ contained in $\cal C$. Observe that we can take $M$ to be a maximal subgroup of $G$ and $N$ to be a minimal subgroup of $G$. Then $M$ contains any minimal subgroup different from $N$, while $N$ is contained in any maximal subgroup different from $M$. These remarks lead to
\begin{itemize}
\item[1.] $\Omega_1(G)\subseteq M$;
\item[2.] $N\subseteq\Phi(G)$ (and, in particular, $N\subseteq M$).
\end{itemize}An example and a non-example of such a group are given by the following two propositions.

\begin{proposition}
    The modular $p$-group $M(p^n)=\langle x,y\,|\,x^{p^{n-1}}=y^p=1, x^y=x^{p^{n-2}+1}\rangle$ belongs to $\cal C$.
\end{proposition}

\begin{proof}
    It is well-known that $M(p^n)'=\langle x^q\rangle$, where $q=p^{n-2}$. Moreover, $M(p^n)/M(p^n)'\cong\mathbb{Z}_p\times\mathbb{Z}_{p^{n-2}}$. We also have $$\Omega_1(M(p^n))=\langle x^q, y\rangle\cong\mathbb{Z}_p\times\mathbb{Z}_p \mbox{ and } \Phi(M(p^n))=\langle x^p\rangle.$$Then $M(p^n)$ contains $p+1$ minimal subgroups and the join of any two distinct such subgroups includes $M(p^n)'$. We infer that $L(M(p^n))$ consist of the interval $[M(p^n)',M(p^n)]$, of $p$ minimal non-normal subgroups, and of the trivial subgroup. By taking $$M=\langle x^p, y\rangle\supseteq\Omega_1(M(p^n)) \mbox{ and } N=\langle x^q\rangle\subseteq\Phi(M(p^n)),$$it follows that $$\overline{L}(M(p^n))=[[1],[M]]\cup[[N],[M(p^n)]],$$as desired.
\end{proof}

\begin{proposition}
    The dihedral $2$-group $D_{2^n}=\langle x,y\,|\,x^{2^{n-1}}=y^2=1, x^y=x^{-1}\rangle$ does not belong to $\cal C$.
\end{proposition}

\begin{proof}
    Assume that $D_{2^n}$ belongs to $\cal C$ and let $M,N$ be two proper non-trivial subgroups of $D_{2^n}$ such that
    $$\overline{L}(D_{2^n})=[[1],[M]]\cup[[N],[D_{2^n}]]$$with $M$ maximal and $N$ minimal. Then $\Omega_1(D_{2^n})\subseteq M$.
    But $y, xy\in\Omega_1(D_{2^n})$, and so $x=(xy)y\in\Omega_1(D_{2^n})$, implying that $\Omega_1(D_{2^n})=D_{2^n}$. Thus
    $M=D_{2^n}$, a contradiction.
\end{proof}

Our next theorem gives a sufficient condition for a finite solvable group which is not a $p$-group to be contained in $\cal C$.

\begin{theorem}
    Let $G$ be a finite solvable group with $|\pi(G)|\geq 2$. If there is a prime $p\in\pi(G)$ such that all subgroups of order
    $p$ of $G$ are conjugate, then $G$ belongs to $\cal C$. In particular, if a Sylow subgroup of $G$ is cyclic or a generalized
    quaternion $2$-group, then $G$ belongs to $\cal C$.
\end{theorem}

\begin{proof}
    Let $p\in\pi(G)$ such that all subgroups of order $p$ of $G$ are conjugate. Pick a subgroup $N\leq G$ of order $p$, and a $p$-complement $M$ in $G$. We will prove that $$\overline{L}(G)=[[1],[M]]\cup[[N],[G]].$$Let $m=|M|$. Then $p\nmid m$, any two subgroups of order $m$ of $G$ are conjugate, and any subgroup of $G$ of order dividing $m$ is contained in a subgroup of order $m$. Let $H\leq G$. If $|H|\mid m$, then $H$ is contained in a conjugate of $M$, that is $[H]\leq [M]$. If $|H|\nmid m$, then $p\mid |H|$ and so $H$ contains a subgroup of order $p$ of $G$. By hypothesis, there is $x\in G$ such that $N^x\subseteq H$, or equivalently $N\subseteq H^{x^{-1}}$, proving that $[N]\leq [H]$.
\end{proof}

Note that the condition in Theorem 6 is not necessary for a finite solvable group $G$ to be contained in $\cal C$: for example, $G=\mathbb{Z}_2^2\times M(3^3)$ belongs to $\cal C$ (see Theorem 9 below), but its subgroups of a fixed prime order are not conjugate. Also, we cannot remove the hypothesis that $G$ is solvable: for example, $G=A_5$ has all subgroups of order $3$ (as well as all subgroups of order $5$) conjugate, but it does not belong to $\cal C$.
\medskip

An important class of groups contained in $\cal C$ is indicated in the following corollary. Recall that a finite group is called a \textit{ZM-group} if all its Sylow subgroups are cyclic.

\begin{corollary}
    Any ZM-group belongs to $\cal C$.
\end{corollary}

Also, by using Theorem 6, we are able to determine all finite dihedral groups contained in $\cal C$.

\begin{corollary}
    The dihedral group $D_{2n}$ belongs to $\cal C$ if and only if $n$ is not a power of $2$.
\end{corollary}

\begin{proof}
    Let $n=2^km$, where $2\nmid m$. According to Proposition 5, it suffices to show that  $D_{2n}$ belongs to $\cal C$ if $m\neq 1$. Let $p$ be a prime diving $m$. Since $p$ is odd, $D_{2n}$ has a unique Sylow $p$-subgroup and this is cyclic. Therefore the conclusion follows from Theorem 6.
\end{proof}

    Finally, we remark that for finite nilpotent groups the condition in Theorem 6 can be weakened.
    
\begin{theorem}
    Let $G$ be a finite nilpotent group. If there is a prime $p\in\pi(G)$ such that a Sylow $p$-subgroup of $G$ belongs to $\cal C$, 
    then $G$ also belongs to $\cal C$.
\end{theorem}

\begin{proof}
    Since a finite nilpotent group is the direct product of its Sylow $p$-subgroups, it suffices to prove that if $G_1$, $G_2$ are two finite groups of coprime orders and $G_1$ belongs to $\cal C$, then $G=G_1\times G_2$ also belongs to $\cal C$.
    
    Let $M$ and $N$ be two proper non-trivial subgroups of $G_1$ such that $$\overline{L}(G_1)=[[1],[M]]\cup[[N],[G_1]].$$This implies that $$\overline{L}(G)=[[1],[M\times G_2]]\cup[[N\times 1],[G]].$$Indeed, any subgroup $H$ of $G$ can be written as $H=H_1\times H_2$ with $H_i\leq G_i$, $i=1,2$. We easily infer that if $[H_1]\leq[M]$ then $[H]\leq[M\times G_2]$, while if $[N]\leq[H_1]$ then $[N\times 1]\leq[H]$, completing the proof.
\end{proof}

Note that the condition in Theorem 9 is also not necessary for a finite nilpotent group $G$ to be contained in $\cal C$: for example, $G=Z_6$ belongs to $\cal C$, in contrast with its Sylow subgroups $\mathbb{Z}_2$ and $\mathbb{Z}_3$.
\medskip

    We end our paper by indicating a natural open problem concerning the above study.

\medskip\noindent{\bf Open problem.} Characterize the finite groups which are contained in the class $\cal C$.

\vspace*{5ex}\small

\hfill
\begin{minipage}[t]{5cm}
Marius T\u arn\u auceanu \\
Faculty of  Mathematics \\
``Al.I. Cuza'' University \\
Ia\c si, Romania \\
e-mail: {\tt tarnauc@uaic.ro}
\end{minipage}


\begin{thebibliography}{10}
\bibitem{1} A. Breaz and G. Calugareanu, {\it Abelian groups whose subgroup lattice is the union of two intervals}, J. Aust. Math. Soc., vol. 78 (2005), no. 1, pp. 27-36.
\bibitem{2} G. C\u alug\u areanu and M. Deaconescu, {\it Breaking points in subgroup lattices}, Proceedings of Groups St. Andrews 2001 in Oxford, vol. 1, Cambridge University Press, 2003, pp. 59-62.
\bibitem{3} Y. Chen and G. Chen, {\it A note on a characterization of generalized quaternion $2$-groups}, C. R. Math. Acad. Sci. Paris, vol. 352 (2014), no. 6, pp. 459-461.
\bibitem{4} I.M. Isaacs, {\it Finite group theory}, Amer. Math. Soc., Providence, R.I., 2008.
\bibitem{5} R. Schmidt, {\it Subgroup lattices of groups}, de Gruyter Expositions in Ma\-the\-matics 14, de Gruyter, Berlin, 1994.
\bibitem{6} M. Suzuki, {\it Group theory}, I, II, Springer Verlag, Berlin, 1982, 1986.
\bibitem{7} M. T\u arn\u auceanu, {\it A characterization of generalized quaternion $2$-groups}, C. R. Math. Acad. Sci. Paris, vol. 348 (2010), no. 13-14, pp. 731-733.
\bibitem{8} M. T\u arn\u auceanu, {\it Contributions to the study of subgroup lattices}, Ed. Matrix Rom, Bucure\c sti, 2016.
\end{thebibliography}
\end{document}